\documentclass{amsart}
\input xypic
\usepackage{amssymb}   
\usepackage{amsmath}
\usepackage{amsthm}
\usepackage{times}
\newtheorem{Theorem}{Theorem}

\newtheorem{Proposition}[Theorem]{Proposition}
\newtheorem{Corollary}[Theorem]{Corollary}

\theoremstyle{definition}

\theoremstyle{definition}
\newtheorem{Definition}[Theorem]{Definition}
\newtheorem{Remark}[Theorem]{Remark}
\newtheorem{Example}[Theorem]{Example}

\numberwithin{Theorem}{section} 
\numberwithin{equation}{section}

\newcommand{\Pic}{\operatorname{Pic}}
\newcommand{\Bl}{\operatorname{Bl}}
\newcommand{\Sing}{\operatorname{Sing}}

\renewcommand{\O}{{\mathcal O}}

\newcommand{\I}{{\mathcal I}}

\newcommand{\p}{{\mathbb P}}

\newcommand{\codim}{\operatorname{codim}}

\newcommand{\red}{\operatorname{red}}

\newcommand{\map}{\dasharrow}

\newcommand{\Supp}{\operatorname{Supp}}

\newcommand{\reg}{\operatorname{reg}}

\begin{document}

\begin{abstract}
The first part of this note contains a review of basic properties 
of the variety of lines contained in an embedded  projective 
variety  and passing through a general point. In particular we provide
a detailed proof that  for varieties defined by quadratic equations the base
locus of the projective second fundamental form at a general point coincides, as a scheme,
with the variety of lines.

The second part concerns the problem of extending embedded projective manifolds, 
using the geometry of the variety of lines. Some applications to the case of homogeneous
manifolds are included.

\end{abstract}

\title[Lines on  projective varieties ]{Lines on projective varieties and applications}
\author{Francesco Russo}
\address{Dipartimento di Matematica,
Universit\`a di Catania,
Viale A. Doria 6, 95125 Catania, Italy}
\email{frusso@dmi.unict.it}

\maketitle

\section*{Introduction}

 The {\it principle} that the Hilbert scheme of 
lines contained in a (smooth) projective
variety $X\subset\p^N$ and passing through a (general) point can inherit intrinsic and extrinsic geometrical properties 
of the variety, has emerged recently. This principle allowed
to attack some  problems in a {\it unified way}, provided 
non trivial connections between different theories and put some basic questions in a new light. A typical example is
the  Hartshorne Conjecture on complete intersections, see \cite{HC, DD} and also \cite{QEL1, QEL2}. The technique
of studying, or even reconstructing, $X$ from the {\it variety of minimal rational tangents} introduced 
 in  the work of Hwang, Mok and others  (a generalization of the Hilbert scheme of lines
passing through a point) was applied to the theory of Fano manifolds  (see e.g. \cite{HM,  HM2, HM3, Hwang, HK, FHw}).   On the other hand, Landsberg and others investigated some  possible characterizations of  special homogeneous
manifolds via the projective second fundamental form (see e.g. \cite{Lan3, Lan4, HY, LR}). 

The Hilbert schemes of lines through a general point
of many homogeneous varieties with notable geometrical properties  are also somehow {\it nested}, see Tables \eqref{hermitian} 
and \eqref{contact}, or {\it part of a matrioska}. For this class of varieties, or more generally for classes
where  the principle  holds,
one  
starts an induction process which sometimes stops after only a few steps, see e.g.
\cite[Theorem 2.8, Corollary 3.1 and 3.2]{QEL1} and also \cite{FHw}. 
An example of this kind is the following: if $X\subset\p^N$ is a $LQEL$-manifold of type $\delta\geq 3$, then the Hilbert scheme of lines $\mathcal L_{x,X}\subset\p^{n-1}$, $n=\dim(X)$, passing through a general point $x\in X$ 
is a $QEL$-manifold of type $\delta-2$, \cite[Theorem 2.3]{QEL1}. Then starting the induction with $X\subset\p^N$ a $LQEL$-manifold of type $\frac{n}{2}$, one deduces
immediately $n=2,4,8$ or 16, yielding as a consequence a quick proof that Severi varieties appear only in  these dimensions (see \cite[Corollary 3.2]{QEL1}, also for the definitions of $(L)QEL$-variety and of Severi variety, introduced by   Zak, see e.g. \cite{Zak}).

The Hilbert scheme of lines through a point is closely related to  the base locus of the (projective) second fundamental form, a classical tool used in  projective differential geometry and reconsidered 
in modern algebraic geometry by Griffiths and Harris (see  \cite{GH} and also \cite{IL}). In this theory  one tries to  reconstruct a (homogeneous) variety
from its second fundamental form (see e.g \cite{Lan3, Lan4, HY, LR})  by integrating local differential equations and  obtaining global results. We note that the base locus of the second fundamental form at a general point of a smooth variety
is typically not smooth, while this property is preserved by the Hilbert scheme of lines, see Proposition \ref{Yx}.
                              
An important class where the two previous objects coincide is that of {\it quadratic varieties},
that is varieties $X\subset\p^N$ scheme theoretically defined by quadratic equations.
All known {\it prime Fano  manifolds of high index},
 other than complete intersections (for example many homogeneous manifolds), are quadratic; moreover, they are embedded with {\it small codimension}.
For quadratic varieties  the Hilbert scheme of lines through a smooth  point is also quadratic,
see Proposition \ref{quadraticLx}. Moreover,  since it coincides with the base locus scheme of the second fundamental form, it may be scheme theoretically defined by at most $c=\codim(X)$
(quadratic) equations, see Corollary  \ref{quadraticformLx}. 
If $X\subset\p^N$ is smooth and $x\in X$ is general, 
then $\mathcal L_{x,X}\subset\p^{n-1}$ is also smooth, see Proposition \ref{Yx}. Thus for quadratic manifolds, if $\mathcal L_{x,X}$ is also irreducible,  a beautiful  {\it matrioska} naturally appears.
From this point of view, a quadratic manifold
$X\subset\p^N$ with $3n>2N$  is a complete intersection
{\it because} $\mathcal L_{x,X}\subset\p^{n-1}$ is a smooth irreducible non-degenerate
complete intersection, defined exactly by $c$ quadratic equations, so that  it has  
the {\it right dimension}, \cite[Theorems 4.8 and 2.4]{DD} and Remark \ref{ci}.

The aim of this note is twofold:  
In \S \ref{Lx} we  study in detail the intrinsic and extrinsic properties of the Hilbert
scheme of lines passing through a smooth point of an equidimensional connected variety
$X\subset\p^N$, providing  an almost  self contained  treatment. In \S 2 we illustrate another incarnation of the principle presented above by studying the problem of extending smooth varieties uniruled by lines as hyperplane sections of irreducible varieties.

First we describe the possible singularities of $\mathcal L_{x,X}$,
proving that a singular point of the Hilbert scheme of lines 
passing through a general point $x$ of an irreducible variety produces a line joining $x$ to a singular point of $X$, a stronger condition 
than the mere existence of a singular point on $X$, see Proposition \ref{Yx}. Then
we relate the equations defining $X\subset\p^N$ with those of 
$\mathcal L_{x,X}\subset\p((t_{x,X})^*)$, see \eqref{eqLxE}. This is applied to quadratic varieties showing that the Hilbert scheme of lines passing through a smooth point is 
a quadratic scheme, which coincides with the projectivized tangent cone at $x$ to the scheme $T_xX\cap X$,
see Proposition \ref{quadraticLx}. After introducing the base locus of the second fundamental
form of $X$ at $x$, $B_{x,X}\subset \p((t_{x,X})^*)$, we show that in general 
$\mathcal L_{x,X}\subseteq B_{x,X}$ as schemes with equality holding, as schemes, if  $X\subset\p^N$ is quadratic, see Corollary
\ref{quadraticformLx},  
 \cite[Theorem 2.4 and \S 4]
{DD} and also Proposition \ref{settheoretically} here. Then we recall some results about lines on prime Fano manifolds to illustrate further
how geometric properties of $X\subset\p^N$ are transferred  to $\mathcal L_{x,X}\subset\p((t_xX)^*)$, see Proposition \ref{Fano}
and Example \ref{excomp}.

In \S \ref{ext} we  consider the classical problem of 
the existence of projective extensions  $X\subset\p^{N+1}$ 
of a  subvariety $Y\subset \p^N\subset\p^{N+1}$. 
It is well known that some special manifolds cannot be hyperplane sections of smooth varieties
and that in some cases only the trivial extensions exist. These are given by cones over
$Y$ with vertex a point $p\in\p^{N+1}\setminus\p^N$ (see e.g. \cite{CSegre}, \cite{Scorza1}, \cite{Scorza2},
\cite{Terracini} and also \S \ref{ext} for precise definitions).
Recently the interest in the above problem 
 (and further  generalizations of it) was renewed.
 Complete references, many results and 
a lot of interesting connections with other
areas, such as  deformation theory of isolated singularities, can be found in the monograph \cite{Badescu}, especially relevant for this
problem being Chapters 1 and 5. One could also look at the survey \cite{Zakdual}. 

Many sufficient conditions for the non-existence of non-trivial extensions of smooth varieties are known.
These conditions are usually expressed, in the more general setting of extensions as ample divisors, by the vanishing of (infinitely many) cohomology groups
of the twisted tangent bundle of $Y$ (or of its normal bundle in $\p^N$). These results are general and concern a lot of applications, see {\it loc. cit.}, but even in the simplest cases the computation of these
cohomology groups can be quite complicated. In any case their geometrical meaning is not so obvious to the non-expert in the field.

Here we prove a simple geometrical sufficient condition for non-extendability, Theorem \ref{criterion}, for  smooth projective
complex varieties uniruled by lines. The simplest version
states that $Y\subset\p^N$ 
admits only trivial extensions $X\subset\p^{N+1}$ as soon as  $\mathcal L_{y,X}\subset  \p((t_yX)^*)$ admits no smooth extension (a weaker condition than the thesis!). Indeed, one easily shows in Proposition \ref{extlines}, via the results of \S 1, that also
$\mathcal L_{y, X}\subset\p((t_yX)^*)$ is a projective extension of  
$\mathcal L_{y,Y}\subset\p((t_yY)^*)$ for $y\in Y$ general. Then under the hypothesis of Theorem \ref{criterion}  one deduces  the existence of a line through $y$ and a singular point $p_y\in X$. Then $p_y=p$ does not vary with $y\in Y$ general since $X$ has at most a finite number of singular points so that  $X\subset\p^{N+1}$ is a cone of vertex $p$.
The range of applications of Theorem \ref{criterion}
is quite wide, see Corollary \ref{Segre}, \ref{exthermitian}, \ref{extcontact}, allowing us  to recover
some results previously obtained differently, see \cite{Badescu} and Remark \ref{dualhom}.

We were led to the analysis of the problem of extending smooth varieties by the desire of understanding geometrically why in some well-known examples the geometry of $Y\subset\p^N$ forces that every extension is trivial and
by the curiosity of explicitly constructing the cones extending $Y$.
Moreover, this approach reveals that  Scorza's  result about the non-extendability of $\p^a\times\p^b\subset \p^{ab+a+b}$ for $a+b\geq 3$, originally proved in \cite{Scorza2} and recovered later by many authors (see e.g. \cite{Badescu} and Corollary \ref{Segre} here), implies the non-extendability of a lot of homogeneous varieties via the
description of their Hilbert scheme of lines. 
From this perspective the  Pl\" ucker embedding of $Y=\mathbb G(r,m)$, with $1\leq r<m-1$ and for $r=1$
with $m\geq 4$, admits only trivial extensions because  $\mathcal L_{y,Y}=\p^r\times\p^{m-r-1}$  admits only trivial extensions (see \cite{Fiore} for an ad-hoc proof following Scorza's approach).  
Besides the  applications contained in  Corollary  \ref{Segre} and \ref{exthermitian}, we 
also show that our analysis can be used to provide a direct proof that 
$\nu_2(\p^n)\subset\p^{\frac{n(n+3)}{2}}$ admits only trivial extensions, see Proposition \ref{Veronese}, a well-known classical fact originally proved by Scorza in \cite{Scorza1} and later obtained differently by many authors.

{\bf Acknowledgements}. 
I am indebted to Paltin Ionescu  for his  useful remarks and comments leading to an improvement
of the exposition and especially for various discussions on these subjects.  Giovanni Staglian\` o read carefully the text
and made useful comments on a preliminary version.
A special thank  to  Prof. Markus Brodmann for 
his invitation  to give a talk at the Oberseminar at Z\" urich University in May 2010,  for his kind hospitality and for his interest in my work. On that  occasion I began to organize the material contained in \S 1.

\section{Geometry of (the  Hilbert scheme of) lines contained in a variety and passing through a (general) point}\label{Lx}

\subsection{Notation, definitions and preliminary results}\label{prel}
Let $X\subset\p^N$ be a (non-degenerate) connected  equidimensional projective variety of dimension $n\geq 1$,
defined over a fixed algebraically closed field of characteristic zero, which
from now on will be simply called a {\it projective variety}. If $X$ is smooth and irreducible, we shall call $X$ a {\it manifold}.
Let $X_{\reg}=X\setminus\Sing(X)$ be the smooth locus of $X$. Let $t_xX$ denote {\it the affine tangent space to $X$ at $x$}, let $T_xX\subset\p^N$ denote
{\it the projective tangent space to $X$ at $x$ of $X\subset\p^N$} and for an arbitrary scheme
$Z$ and for a closed point $z\in Z$ let $C_zZ$ denote {\it the affine tangent cone to $Z$ at $z$}. Let $\mathcal L_{x,X}$
denote  the Hilbert scheme of lines contained in $X$ and passing through the point $x\in X$. For  a line $L\subset X$
passing through $x$, we let $[L]\in \mathcal L_{x,X}$ be the corresponding point.

Let $\pi_x: \mathcal{H}_x \to \mathcal{L}_{x,X}$ denote the universal family 
and let $\phi_x: \mathcal{H}_x \to X$ be the tautological morphism. 
From now on we shall always suppose that $x\in X_{\reg}$.
Note that $\pi_x$ admits a section $s_x:\mathcal L_{x,X}\to {\mathcal E}_x\subset\mathcal H_x$, which is contracted by $\phi_x$ to the point $x$.  Consider the blowing-up $\sigma_x: \Bl_xX\to X$ of $X$ at $x$. For every $[L]\in\mathcal L_{x,X}$ the line $L=\phi_x(\pi_x^{-1}([L]))$ is smooth at $x$ so that \cite[Lemma 4.3]{IN} and the universal property of the blowing-up ensure the existence of  a morphism $\psi_x : {\mathcal H}_x \to \Bl_xX$ such that $\sigma_x \circ \psi_x = \phi_x$. So we have the following diagram
\begin{equation}\label{joindiagram1}\raisebox{.7cm}{\xymatrix{
&\mathcal{H}_x  \ar[d]_{\pi_x} \ar[dr]^{\phi_x}\ar[r]^{\psi_x}&\Bl_xX\ar[d]^{\sigma_x}\\
&\mathcal{L}_{x,X}&X.
}}
\end{equation}
In particular, $\psi_x$ maps the section ${\mathcal E}_x$   
to $E_x$, the exceptional divisor of $\sigma_x$. 
Let $\widetilde \psi _x : {\mathcal E}_x \to E_x$ be the restriction of $\psi_x$ 
to $\mathcal E_x$. 
We can define the  morphism 
\begin{equation}\label{taux}
\tau_x=\tau_{x,X}=\widetilde \psi_x\circ s_x:\mathcal L_{x,X}\to \p((t_xX)^*)=E_x=\p^{n-1},
\end{equation}
which associates to each line $[L]\in \mathcal L_{x,X}$ the corresponding tangent direction through
$x$, i.e. $\tau_x([L])=\p((t_xL)^*)$. The morphism  $\tau_x$ is clearly injective and we claim that $\tau_x$ is a closed immersion.
Indeed, by taking in the previous construction  $X=\p^N$ the corresponding morphism $\tau_{x,\p^N}:\mathcal L_{x,\p^N}\to \p((t_x\p^N)^*)=\p^{N-1}$ is an isomorphism between $\mathcal L_{x,\p^N}$ and the exceptional divisor of $\Bl_x\p^N$. By definition
the inclusion $X\subset\p^N$ induces a closed embedding $i_x:\mathcal L_{x,X}\to\mathcal L_{x,\p^N}$. If $j_x:\p((t_xX)^*)\to
\p((t_x\p^N)^*)$ is the natural closed embedding, then we have the following commutative diagram
\begin{equation}\label{diagram2}{\xymatrix{
&\mathcal{L}_{x,X}  \ar[d]_{i_x} \ar[r]^{\tau_{x,X}}&\p((t_xX)^*)\ar[d]^{j_x}\\
&\mathcal{L}_{x,\p^N}\ar[r]^{\tau_{x,\p^N}}&\p((t_x\p^N)^*),
}}
\end{equation}
proving the claim.

For $x\in X_{\reg}$ such that $\mathcal L_{x,X}\neq\emptyset$, we shall always identify $\mathcal L_{x,X}$ with $\tau_x(\mathcal L_{x,X})$ and we shall naturally  consider $\mathcal L_{x,X}$ as a subscheme of $\p^{n-1}=\p((t_xX)^*)$. 
 We denote by $\mathcal C_x$ the scheme theoretic image of $\mathcal H_x$, that is $\phi_x(\mathcal H_x)=\mathcal C_x\subset X$. Via \eqref{joindiagram1}
we deduce the following relation:
\begin{equation}\label{tangentconeCx}
\p(C_x(\mathcal C_x))=\mathcal L_{x,X},
\end{equation}
as subschemes of $\p((t_xX)^*)$, where $\p(C_x(\mathcal C_x))$
 is the {\it projectivized tangent cone to $\mathcal C_x$ at $x$}, see \cite[II,\S 3]{Mumford}.

\subsection{Singularities of $\mathcal L_{x,X}$}\label{sing}
We begin by studying the  intrinsic geometry of $\mathcal L_{x,X}\subset\p^{n-1}$. When it is clear from the context which variety $X\subset\p^N$ we are considering we shall write $\mathcal L_x$ instead of $\mathcal L_{x,X}$.

The normal bundle $N_{L/X}$ is locally free being a subsheaf of the locally free sheaf $N_{L/\p^N}\simeq\O_{\p^1}(1)^{N-1}$.
If $L\cap X_{\reg}\neq\emptyset$, then $N_{L/X}$ is locally free of rank $n-1$ and more precisely
\begin{equation}\label{ai1}
N_{L/X}\simeq\bigoplus_{i=1}^{n-1} \O_{\p^1}(a_i),
\end{equation}
with $a_i\leq 1$ since  $N_{L/X}\subset N_{L/\p^N}$.

 If $N_{L/X}$ is  
also generated by global sections, then
\begin{equation}\label{split}
N_{L/X}\simeq \O_{\p^1}(1)^{s(L,X)}\oplus\O_{\p^1}^{n-1-s(L,X)}.
\end{equation}
Therefore if $N_{L/X}$ is generated by global sections, then $\mathcal L_x$ is unobstructed at $[L]$, that is $h^1(N_{L/X}(-1))=0$,
$\mathcal L_x$ is smooth at $[L]$ and $\dim_{[L]}(\mathcal L_x)=h^0(N_{L/X}(-1))=s(L,X)$, 
where $s(L,X)\geq 0$ is the integer  defined in \eqref{split}.

For $x\in X_{\reg}$,  let

$$S_x=S_{x,X}=\{[L]\in \mathcal L_x\text{ such that } L\cap \Sing(X)\neq \emptyset\;\}\subseteq \mathcal L_x.$$

Then $S_{x,X}$ has a natural scheme structure and the previous inclusion holds at the scheme theoretic level.
If $X$ is smooth, then  $S_{x,X}=\emptyset$. Moreover, if $L\subset X$ is a line passing through $x\in X_{\reg}$, clearly
 $[L]\not\in S_{x,X}$ if and only if $L\subset X_{\reg}$.
\medskip

We now prove that  a singular point of $\mathcal L_x$
produces a line passing through $x$ and through a singular point of $X$, a stronger condition than
the mere existence of a singular point on $X$. These results are  well known to experts, at least for manifolds,  see \cite[Proposition 1.5]{Hwang} and also  
 \cite[Proposition 2.2]{QEL1}.  In \cite{DG}, the singularities of the Hilbert scheme of lines contained in a projective variety are related to some geometrical
 properties of the variety.
 \medskip

\begin{Proposition}\label{Yx}
Let notation be as above and let $X\subset\p^N$ be an irreducible projective variety of dimension $n\geq 2$. Then for $x\in X_{\reg}$ general:
\begin{enumerate}
\item  $\mathcal L_x\subset\p^{n-1}$ is smooth outside $S_{x,X}$, that is $\Sing(\mathcal L_x)\subseteq S_x.$
In particular if $X\subset\p^N$ is smooth and if $x\in X$ is general, then $\mathcal L_x\subset\p^{n-1}$
is a smooth variety.
\medskip

\item If $\mathcal L_x^j$, $j=1,\ldots,m$, are the irreducible components
of $\mathcal L_x$ and if 
 \[\dim(\mathcal L_x^l)+\dim(\mathcal L_x^p)\geq n-1 \quad \mbox{for some } l\neq
p,\]
then  $\mathcal L_x$ is singular, $X$ is singular and there exists a line $[L]\in \mathcal L_x$ such
that $L\cap \Sing(X)\neq\emptyset$. 
\end{enumerate}
\end{Proposition}
\begin{proof}
There exists an open dense subset $U\subseteq X_{\reg}$ such that for
every line $L\subset X_{\reg}$ such that $L\cap U\neq \emptyset$
the normal bundle $N_{L/X}$ is generated by global sections,
see for example \cite[Proposition 4.14]{Debarre}. Combining this result
with the above discussion, we deduce that for every $x\in U$ the variety
$\mathcal L_x\subset\p^{n-1}$ is smooth outside $S_x$, proving the first assertion.

The condition on
the dimensions of two irreducible components of $\mathcal L_x$ in (2) ensures that these
components  have to  intersect in $\p^{n-1}$. A point of intersection is a singular
point of $\mathcal L_x\subset\p^{n-1}$. This forces $X$ to be singular by the first part and 
also the existence of
a line $[L]\in S_x$, which by definition cuts $\Sing(X)$.
\end{proof}

\subsection{Equations for $\mathcal L_{x,X}\subset\p((t_xX)^*)$}\label{equations}
\medskip

We now  follow  and expand the treatment outlined in \cite[Theorem 2.4]{DD} by looking at the equations defining $\mathcal L_x\subset\p^{n-1}$ for $x\in X_{\reg}$.

Let $$    X=V(f_1,\ldots, f_m)\subset\p^{N}\hskip 3cm (\ast),$$ 
be a projective equidimensional connected variety, not necessarily irreducible, let $x\in X_{\reg}$, let $n=\dim(X)$ and let $c=\codim(X)=N-n$. Thus we are assuming that 
$X\subset\p^N$ is  scheme theoretically the intersection of $m\geq 1$
 hypersurfaces of degrees $d_1\geq d_2\geq\ldots\geq d_m\geq 2.$ Moreover it is implicitly assumed that $m$ is minimal, i.e. none of the hypersurfaces contains the intersection of the others. 
 Define, following \cite{DD}, the integer  
 
 $$d:=\min\{\sum_{i=1}^{c}(d_i-1)\text{ for expressions  $(\ast)$ as above}\}\geq c.$$

  With these definitions  $X\subset\p^N$ (or more generally a scheme $Z\subset\p^N$) is called {\it quadratic} if it is scheme theoretically an intersection of quadrics, which means that  we can assume $d_1=2$. In particular  $X\subset\p^N$ is quadratic if and only if $d=c$.
\medskip

We can choose homogeneous coordinates $(x_0:\ldots:x_N)$ on $\p^N$ such that  $x=(1:0:\ldots:0),$ $T_xX=V(x_{n+1},\ldots, x_N).$ Let $\mathbb A^N=\p^N\setminus V(x_0)$ with affine coordinates $(y_1,\ldots, y_N)$, that is   $y_l=\frac{x_l}{x_0}$ for every $l=1,\ldots, N$. Let $\widetilde \p^N=\Bl_x\p^N$ with exceptional divisor $E'\simeq\p((t_x\p^N)^*)=\p^{N-1}$ and let $\widetilde X=\Bl_xX$ with exceptional
divisor $E=\p((t_xX)^*)=\p^{n-1}$. Looking at the graph of the projection from $x$ onto $V(x_0)$ we can naturally
identify the projectivization of $\mathbb A^N\setminus \mathbf 0=\mathbb A^N\setminus x$ with $E'$ and with the projective hyperplane $V(x_0)=\p^{N-1}$.

Let $f_i=f_i^1+f_i^2+\cdots+f_i^{d_i}$, with $f_i^j$ homogeneous of degree $j$ in the variables $(y_1,\ldots, y_N)$. So
 $f_1^1=\ldots=f_m^1=0$ are the equations of $t_xX=T_xX\cap\mathbb A^N\subset\mathbb A^N$,
which reduce to $y_{n+1}=\ldots=y_N=0$ by the previous choice of coordinates, yielding 
$$V(f_1^1,\cdots, f_m^1)=\p((t_xX)^*)\subset \p(( t_x\p^N)^*)=\p^{N-1}.$$

With the previous identifications  $\mathcal L_{x,\p^N}=E'=\p^{N-1}=\p((t_x\p^N)^*)$. We now write a set of equations defining  $\mathcal L_x\subset E\subset E'$ as a subscheme of $E'$ and of $E$. By definition  $\mathbf y=(y_1:\ldots:y_n)$ are homogeneous coordinates on 
$E\subset E'$. For every $j=2,\ldots, m$ and for every $i=1,\ldots, m$, let 
$$\widetilde f_i^j(\mathbf y)=f_i^j(y_1,\ldots,y_n,0,0,\ldots,0,0).$$

Then we have that $\mathcal L_x\subset E'$ is  the scheme $$
V(f_1^1,f_1^2,\cdots,f_1^{d_1}, \cdots, f_m^1,f_m^2,\cdots,f_m^{d_m})\subset E',$$ while $\mathcal L_x\subset E$ is the scheme
\begin{equation}\label{eqLxE}
V(\widetilde f_1^2,\cdots,\widetilde f_1^{d_1}, \cdots, \widetilde f_m^2,\cdots,\widetilde f_m^{d_m}),
\end{equation}
so that  it is scheme theoretically defined by at most $\sum_{i=1}^m(d_i-1)$ equations.

The equations of $T_xX\cap X\cap \mathbb A^N=t_xX\cap X\cap \mathbb A^N$, as a subscheme of $\mathbb A^N$, are 
$$V(f_1^1, \ldots, f_m^1, f_1^1+f_1^2+\cdots+f_1^{d_1},\ldots, f_m^1+f_m^2+\cdots+f_m^{d_m})=$$
\begin{equation}\label{eqTxAN}
V(f_1^1, \ldots, f_m^1, f_1^2+\cdots+f_1^{d_1},\ldots, f_m^2+\cdots+f_m^{d_m})\subset\mathbb A^N.
\end{equation}
Thus the equations of $T_xX\cap X\cap\mathbb A^N=t_xX\cap X\cap\mathbb A^N$ as a subscheme of $t_x(X\cap \mathbb A^N)=t_xX$ are
\begin{equation}\label{eqTxtx}
V(\widetilde  f_1^2+\cdots+\widetilde f_1^{d_1},\ldots, \widetilde f_m^2+\cdots+\widetilde f_m^{d_m})\subset t_xX=\mathbb A^n.
\end{equation}

Let $I=\langle \widetilde  f_1^2+\cdots+\widetilde f_1^{d_1},\ldots, \widetilde f_m^2+\cdots+\widetilde f_m^{d_m}\rangle\subset\mathbb C[y_1,\ldots,y_n]=S$ and let $I^*$ be the ideal generated by the {\it initial forms} of elements of $I$. Remark  that if $I$ is homogeneous and generated by forms of the same degree, then clearly $I=I^*$. Then the affine tangent cone to $T_xX\cap X$ at $x$ is $C_x(T_xX\cap X)={\rm Spec}(\frac{S}{I^*})$ so that 
\begin{equation}\label{exp1}
\p(C_x(T_xX\cap X))={\rm Proj} (\frac{S}{I^*}),
\end{equation}
see \cite[III, \S\,3]{Mumford}. 

 Let $J\subset S$ be the homogeneous ideal generated by the polynomials in 
\eqref{eqLxE} defining $\mathcal L_{x,X}$ scheme theoretically, that is 
$\mathcal L_{x,X}={\rm Proj}(\frac{S}{J})\subset\p((t_xX)^*)$. Clearly $I^*\subseteq J$, yielding the  closed embedding of schemes 
\begin{equation}\label{inclLxBx}
\mathcal L_{x,X}\subseteq \p(C_x(T_xX\cap X)).
\end{equation}

 If $X\subset\p^N$ is quadratic,  then  $I=I^*=J$. In conclusion we have proved the following results.
\medskip

\begin{Proposition}\label{quadraticLx} Let $X\subset\p^N$ be a (non-degenerate)  projective variety, let $x\in X_{\reg}$ be a  point and let notation be as above.
If $X\subset\p^N$ is quadratic, then

\begin{equation}\label{fund0}
T_xX\cap X\cap\mathbb A^N=C_x(T_xX\cap X)\subset t_xX
\end{equation}
and 

\begin{equation}\label{fund}
\mathcal L_{x,X}=\p(C_x(T_xX\cap X))\subset\p((t_xX)^*).
\end{equation}

In particular if  $X\subset\p^N$ is quadratic, then  the scheme  $\mathcal L_{x,X}\subset\p((t_xX)^*)$
is quadratic.
\end{Proposition}

\subsection{$\mathcal C_x$ versus $T_xX\cap X$ for a quadratic variety}\label{conesex}

The closed embedding \eqref{inclLxBx} holds at the  scheme theoretic level. 
If $\mathcal L_{x,X}$ were reduced, or better smooth, it would be enough to prove that there exists an inclusion as sets. 
Since $x\in X_{\reg}$ was arbitrary we cannot control
 a priori the structure of $\mathcal L_{x,X}$ even if $X\subset\p^N$ is a manifold.
 Recall that by Proposition \ref{Yx}   $\mathcal L_{x,X}$ is smooth as soon as $X$ is a manifold and
$x\in X$ is a general point.

{\it From now on we shall suppose $X\subset\p^N$ quadratic}. Then
\begin{enumerate}
\item $(\mathcal C_x)_{\red}=(T_xX\cap X)_{\red}$; 

\item if $X\subset\p^N$ is a manifold and if
$x\in X$ is a general point, then $\mathcal C_x=(T_xX\cap X)_{\red}$;
\item the strict transforms of $\mathcal C_x$ and of $T_xX\cap X$ on $\Bl_xX$ cut the exceptional
divisor $E=\p((t_xX)^*)$ of $\Bl_xX$ in the same scheme $\mathcal L_{x,X}$ (see \eqref{tangentconeCx} and \eqref{fund});
\item  if $x\in X$ is a general point on
a quadratic manifold $X\subset\p^N$ and if $I^*$ is saturated, then $T_xX\cap X$ is reduced in a neighborhood of $x$ so
that it coincides with $\mathcal C_x$ in a neighborhood of $x$. Indeed since  $T_xX\cap X\cap \mathbb A^n={\rm Spec}(\frac{S}{I})$,
with $I=I^*=J$ homogeneous and saturated, it follows that  $T_xX\cap X$ is reduced at $x$; therefore it is  reduced also in a neighborhood of $x$.
\end{enumerate}

Already for quadratic manifolds there exist many important differences between $\p(C_x(T_xX\cap X))\subset\p((t_xX)^*)$
and $C_x(T_xX\cap X)=T_xX\cap X\cap \mathbb A^N\subset t_xX$ and also between $T_xX\cap X$ and the cone $\mathcal C_x\subseteq T_xX\cap X$. We shall discuss some  examples in order to  analyze closer these  important schemes containing a lot of
geometrical information.
 
\begin{Example} ({\it $T_xX\cap X$ non-reduced only at $x$})\label{exscrolls}
Remark that  $t_x(T_xX\cap X)=t_xX$ so that
$\langle C_x(T_xX\cap X)\rangle=t_xX$, while in some cases $\p(C_x(T_xX\cap X))$
is degenerate in $\p((t_xX)^*)$.  Consider a rational normal scroll  $X\subset\p^N$, different from the Segre varieties $\p^1\times\p^{n-1}$, $n\geq 2$, and a general point $x\in X$. It is well known that $X\subset\p^N$ is quadratic so that $\mathcal L_{x,X}=\p(C_x(T_xX\cap X))\subset\p((t_xX)^*)$ by \eqref{inclLxBx}. On the other hand, if $\p^{n-1}_x$ is the unique $\p^{n-1}$ of the ruling passing through $x\in X$, it is easy to see, letting notation as above, that in this case
$$\mathcal L_{x,X}=\p(\p^{n-1}_x\cap\mathbb A^n)=\p^{n-2}\subset\p((t_xX)^*)=\p^{n-1}.$$
This is possible because in this example $T_xX\cap X$ and $C_x(T_xX\cap X)$ are not reduced at $x$. Indeed, the point $x\in C_x(T_xX\cap X)$  corresponds to the irrelevant ideal of $S$.  $I^*$ is not saturated, because the equation defining the hyperplane $\mathcal L_{x,X}$ belongs to the saturation of $I^*$, but is not in $I^*$ ($I^*$ is generated
by quadratic polynomials!). 
\end{Example}
\medskip

In the case of rational normal scrolls discussed in Example \ref{exscrolls}  we saw that $T_xX\cap X\setminus x=\mathcal C_x\setminus x$ as schemes, the affine tangent cones are different affine schemes, but the projectivized  tangent cones  coincide.

By choosing suitable quadrics $Q_1,\ldots, Q_c$
we shall see in subsection \ref{implicit}  that the complete intersection $Y=Q_1\cap\ldots\cap Q_c$ coincides locally with $X$ around $x$. Thus $T_xY\cap Y$ and $T_xX\cap X$ coincide locally around $x$.
In particular the intersection of their strict transform on $\Bl_xX$ with the exceptional
divisor is the same, so that $\mathcal L_{x,X}=\mathcal L_{x,Y}$ and the last scheme
can be defined scheme theoretically by $r\leq c$ linearly independent quadrics by \eqref{eqLxE}.

In any case the double nature of $T_xX\cap X$ as a subscheme of $T_xX$ and $X$
plays a central role for its infinitesimal properties at $x$, measured exactly by
$\p(C_x(T_xX\cap X))\subset\p((t_xX)^*)$.

It is useful to think of $\p(C_x(T_xX\cap X))\subset\p((t_xX)^*)$ as being the base locus scheme of the restriction to the exceptional divisor over $x$ of the projection of $X$ from $T_xX$,  as we shall 
do in the next section. We shall provide  in this way another reason why $\mathcal L_{x,X}$
can be defined scheme theoretically by at most $c$ quadratic equations for an arbitrary point
$x\in X_{\reg}$.

\subsection{Tangential projection and second fundamental form}\label{second}

There are several possible equivalent
definitions of the projective second fundamental form
$|II_{x,X}|\subseteq\p(S^2(t_xX))$ of a connected equidimensional
projective variety $X\subset\p^N$ at  $x\in X_{\reg}$, see
for example \cite[3.2 and end of Section~3.5]{IL}. We use the
one related to tangential projections, as in \cite[Remark
3.2.11]{IL}.

 Suppose $X\subset\p^N$ is non-degenerate, as always,
 let $x\in X_{\reg}$ and
consider the projection from $T_xX$ onto a disjoint $\p^{c-1}$
\begin{equation}\label{tangentdef}
\pi_x:X\map W_x\subseteq\p^{c-1}.
\end{equation}
The map $\pi_x$ is not defined along the scheme $T_xX\cap X$, which contains $x$,
and it is associated to the linear system of hyperplane
sections cut out by hyperplanes containing $T_xX$, or equivalently
by the hyperplane sections singular~at~$x$.

Let $\phi:\Bl_xX\to X$ be the blow-up of $X$ at $x$, let
\[E=\p((t_xX)^*)=\p^{n-1}\subset\Bl_xX\] be the exceptional
divisor and let $H$ be a hyperplane section of $X\subset\p^N$. The
induced rational map $\widetilde{\pi}_x:\Bl_xX\map\p^{c-1}$ is
defined as a rational map along $E$ since $X\subset\p^N$ is not a linear space, 
see also the discussion below.
 
The restriction of
$\widetilde{\pi}_x$ to $E$ is given by a linear system in
$|\phi^*(H)-2E|_{|E}\subseteq|-2E_{|E}|=|\O_{\p((t_xX)^*)}(2)|=\p(S^2(t_xX))$, whose base locus
scheme will be denoted by $B_{x,X}$.

Consider the strict transform scheme of $T_xX\cap X$ on $\Bl_xX$, denoted from now on by 
$\widetilde T=\Bl_x(T_xX\cap X)$. Then $\widetilde T$ is the base locus scheme of $\widetilde{\pi}_x$ and the restriction of
$\widetilde{\pi}_x$ to $E$ has base locus scheme equal to
\begin{equation}\label{Base}
\widetilde T\cap E=\p(C_x(T_xX\cap X))=B_{x,X}\subset \p((t_xX)^*).
\end{equation}

\begin{Definition}
 The {\it second fundamental form}
$|II_{x,X}|\subseteq\p(S^2(t_xX))$ of a connected equidimensional
non-degenerate projective variety $X\subset\p^N$ of dimension $n\geq 2$ at a
point $x\in X_{\reg}$ is the non-empty linear system of quadric
hypersurfaces in $\p((t_xX)^*)$ defining the restriction of
$\widetilde{\pi}_x$ to $E$ and $B_{x,X}\subset\p((t_xX)^*)$ is the so called {\it base locus scheme
of the second fundamental form of $X$ at $x$}.
\end{Definition}
 
Clearly  $\dim(|II_{x,X}|)\leq c-1$ and
$\widetilde{\pi}_x(E)\subseteq W_x\subseteq\p^{c-1}$. 
Let $\widetilde I\subset S$ be the homogeneous ideal generated by the $r\leq c$ linearly independent quadratic forms in the second fundamental form of $X$ at $x$. Then via \eqref{Base} we obtain  
\begin{equation}\label{projBx}
{\rm Proj (\frac{S}{\widetilde I})}=B_{x,X}=\p(C_x(T_xX\cap X))={\rm Proj (\frac{S}{I^*})}\subset\p((t_xX)^*).
\end{equation}

In conclusion we have proved the following results by combining \eqref{Base} with \eqref{fund} and \eqref{projBx}.
\medskip

\begin{Corollary}\label{quadraticformLx} Let $X\subset\p^N$ be a non-degenerate projective variety, let $x\in X_{\reg}$ be a  point and let notation be as above. Then:
\begin{enumerate}
\item $\mathcal L_{x,X}\subseteq B_{x,X}$;

\item if $X\subset\p^N$ is quadratic, then equality holds and 
 $\mathcal L_{x,X}\subset\p((t_xX)^*)$ can be defined scheme theoretically by the $r\leq c$ quadratic equations defining the second fundamental form of $X$ at $x$.
\end{enumerate} 
\end{Corollary}
\medskip

\begin{Remark}\label{ci} The previous result has many important applications. We recall that, as proved in \cite{DD}, if $X\subset\p^N$ is 
a quadratic manifold and if $c\leq\frac{n-1}{2}$, then, for $x\in X$ general,  $\mathcal L_{x,X}\subset\p((t_xX)^*)$ is the complete intersection
of the $c$ linearly independent quadratic polynomials defining $|II_{x,X}|$. Then $\mathcal L_{x,X}$ has dimension $n-1-c$
from which it follows that $X\subset\p^N$ is a complete intersection. This proves the Hartshorne Conjecture
on complete intersections in the quadratic case and also leads to the classification of quadratic Hartshorne manifolds, see \cite[Theorem 2.4 and Section 4]{DD} for details. 

The paper \cite{PR} considers also  irreducible projective varieties $X\subset\p^{2n+1}$ which are
3--covered by twisted cubics, i.e.  such that through three general points of $X\subset\p^{2n+1}$ there passes a twisted cubic contained in $X$. A key remark for the classification of these varieties  is \cite[Theorem 5.2]{PR}, which among other things shows that for such an $X$  the equality   $\mathcal L_{x,X}=B_{x,X}$ holds for $x\in X$ general.
A posteriori all the known examples of varieties 3--covered by twisted cubics are  projectively equivalent to the so called {\it twisted cubics over Jordan algebras}, which are quadratic, see {\it loc. cit} for definitions and details.
This fact has also many important consequences for the theory of Jordan algebras and for the classification of {\it quadro-quadric}
Cremona transformations, as  shown  in the forthcoming paper \cite{PR2}. 
\end{Remark}

\subsection{Approach to $B_{x,X}=\mathcal L_{x,X}$ via \cite{BEL}}\label{implicit}
For manifolds $X\subset\p^N$ there is another approach 
based on a construction of \cite{BEL} elaborating and generalizing an idea due to Severi, see {\it loc. cit.}
It can be used to give a  proof of a weaker form of Corollary \ref{quadraticformLx}
(in the sense that we shall prove it only for $x\in X$ general); this approach illustrates  the local nature of the second fundamental form.
Let us remark that the treatment in the general setting developed in the previous sections 
is unavoidable because the point $x\in X$ is not necessarily general on the complete intersection $Y\supseteq X$ we now construct.

It was proved in \cite{BEL} that given a manifold $X=V(f_1,\ldots, f_m)\subset\p^N$ as above,  we can choose $g_i\in H^0(\I_X(d_i))$, $i=1,\ldots, c$ such that
\begin{equation}\label{YX}
Y=V(g_1,\ldots, g_c)=X\cup X',
\end{equation}
where $X'$ (if nonempty) meets $X$ in a divisor $D$. Moreover from \eqref{YX} it follows 
\begin{equation}\label{chernD}
\O_X(D)\simeq\det(\frac{\I_X}{\I_X^2})\otimes\O_X(\sum_{i=1}^{c}d_i)\simeq
\O_X(d-n-1)\otimes\omega_X^*,
\end{equation}
see also \cite[pg. 597]{BEL}. We now illustrate the usefulness of this construction 
by proving some  facts and  results contained 
in \cite[Theorem 2.4]{DD}.
\medskip

Suppose that $X\subset \p^N$ is a quadratic manifold and consider a  point $x\in U=X\setminus\Supp(D)$. By definition $Y\setminus\Supp (D)=U\amalg V$, where $V=X'\setminus\Supp(D)$.
Consider the two schemes $T_xX\cap X\cap U$ and $T_xY\cap Y\cap U$. Since $t_xX=t_xY$ and since $Y\cap U=X\cap U$
by the above construction, we obtain the equality as schemes
$$C_x(T_xX\cap X)=C_x(T_xX\cap X\cap U)=C_x(T_xY\cap Y\cap U)=C_x(T_xY\cap Y).$$
Via \eqref{fund} we deduce the following equality as subschemes of $\p((t_xX)^*)$: 
\begin{equation}\label{asinDD} 
\mathcal L_{x,Y}=\p(C_x(T_xY\cap Y))=\p(C_x(T_xX\cap X))=\mathcal L_{x,X}.
\end{equation}

Since $\mathcal L_{x,Y}$ can be scheme-theoretically defined by $r\leq c$ linearly independent quadratic equations, the same is true for $\mathcal L_{x,X}$. 
Now, without assuming anymore that $X$ is quadratic, since $x\in X$ is general, $\mathcal L_{x,X}$ is smooth and hence reduced. Clearly a line $L$ passing through $x$ is contained in $X$
if and only if it is contained in $Y$, yielding $\mathcal L_{x,X}=(\mathcal L_{x,Y})_{\red}$, see \cite[Theorem 2.4]{DD}. We proved:
\medskip

\begin{Proposition}\label{settheoretically} Let $X\subset\p^N$ be a  manifold, let notation be as above and let $x\in U$ be a general point. Then:
\begin{enumerate}
\item $\mathcal L_{x,X}=(\mathcal L_{x,Y})_{\red}$ so that $\mathcal L_{x,X}$ can be defined
set theoretically by the  $r\leq d$ equations defining $\mathcal L_{x,Y}$ scheme theoretically.
In particular, if $d\leq n-1$, then $\mathcal L_{x,X}\neq \emptyset$.
\item If $X\subset\p^N$ is quadratic, then
$\mathcal L_{x,X}=\mathcal L_{x,Y}$ so that $\mathcal L_{x,X}\subset\p((t_xX)^*)$ is a quadratic manifold  defined
scheme theoretically by  at most  $c$ quadratic equations.
\end{enumerate}
\end{Proposition}

\subsection{Lines on prime Fano manifolds}\label{Fano}

Let $X\subset\p^N$ be a (non-degenerate) manifold of dimension $n\geq 2$.
For a general point $x\in X$ we know that $\mathcal L_x\subset\p^{n-1}$ is smooth, Proposition \ref{Yx}.
There are well-known examples when $\mathcal L_x\subset\p^{n-1}$ is not irreducible, such as $X=\p^a\times\p^b\subset\p^{ab+a+b}$ Segre embedded, and also examples where $\mathcal L_x\subset\p^{n-1}$ is degenerate,
see Example \ref{exscrolls} and also table \eqref{contact} below. A relevant class of manifolds where the properties of smoothness, irreducibility
 and non-degeneracy of $X\subset\p^N$ are transfered to $\mathcal L_x\subset\p^{n-1}$ consists of prime Fano manifolds of high index, which we now define.

A manifold $X\subset\p^N$ is called a {\it prime Fano manifold} if $-K_X$ is ample and if 
$\Pic(X)\simeq\mathbb Z\langle\O(1)\rangle$. The {\it index of $X$} is the positive integer
defined by $-K_X=i(X)H$, with $H$ a hyperplane section of $X\subset\p^N$.

Let us recall some fundamental facts. Part (1) below  is well known and follows from the previous discussion except for a fundamental
Theorem of Mori which implies that for prime Fano manifolds of index greater than $\frac{n+1}{2}$, necessarily $\mathcal L_x\neq\emptyset$, see \cite{Mori} and \cite[Theorem V.1.6]{Kollar}.
\medskip

\begin{Proposition} Let $X\subset\p^N$ be a projective manifold and let $x\in X$ be a general 
point. Then 
\begin{enumerate}

\item If $\mathcal L_x\neq\emptyset$, then for every  $[L]\in\mathcal L_x$ 
we have  $\dim_{[L]}(\mathcal L_x)=-K_X\cdot L-2.$ 
In particular for prime Fano manifolds of index $i(X)\geq \frac{n+3}{2}$ 
the variety $\mathcal L_x\subset\p^{n-1}$ is irreducible (and in particular non-empty!).
 
\item {\rm (\cite{Hwang})} If $X\subset\p^N$ is a prime Fano manifold 
of index $i(X)\geq \frac{n+3}{2}$, then $\mathcal L_x\subset\p^{n-1}$ is a  
non-degenerate manifold of dimension $i(X)-2$.
\end{enumerate}
\end{Proposition}
\medskip

Let us finish this section by looking at another significant example in which  meaningful geometrical properties of $X\subset\p^N$ are  reflected in  
similar properties of $\mathcal L_x\subset\p^{n-1}$, when this is non-empty.
\medskip

\begin{Example}\label{excomp}
Let  $X\subset\p^N$ be a smooth complete intersection of type  $(d_1,d_2,\ldots,d_c)$ with  $d_c\geq 2$. Then:

\begin{itemize}

\item if $n+1-d>0,$ then $X$ is a Fano manifold and $i(X)=n+1-d$;

\item if $n\geq 3$, then  $\Pic(X)\simeq\mathbb Z\langle\O(1)\rangle$;

\item if $i(X)\geq 2$, then $\mathcal L_x\neq\emptyset$ and for every $[L]\in\mathcal L_x$ we have $$\dim_{[L]}(\mathcal L_x)=(-K_X\cdot L)-2=i(X)-2=n-1-d\geq 0 ,$$ so  that $\mathcal L_x\subset\p^{n-1}$ is a smooth complete intersection of type $$(2,\ldots,d_1; 
2,\ldots, d_2;\ldots;2,\ldots d_{c-1}; 2,\ldots,d_c)$$ since it is scheme theoretically defined by the $d$ equations in \eqref{eqLxE}.

\end{itemize}
\end{Example}

\section{A condition for non-extendability}\label{ext}
\begin{Definition} Let us consider $H=\p^N$ as a hyperplane in $\p^{N+1}$. Let $Y\subset\p^N=H$
be a smooth (non-degenerate) irreducible variety of dimension $n\geq 1$. An irreducible variety
$X\subset\p^{N+1}$ will be called {\it an extension of $Y$} if 
\medskip

\begin{enumerate}
\item $\dim(X)=\dim(Y)+1$;
\medskip

\item $Y=X\cap H$ as a scheme.
\end{enumerate}
\end{Definition}
\medskip

For every $p\in \p^{N+1}\setminus H$, the irreducible cone 
$$X=S(p,Y)=\bigcup_{y\in Y}<p,y>\subset\p^{N+1}$$
is an extension of $Y\subset\p^N=H$, which will be called {\it trivial}. Let us observe that for any extension $X\subset\p^{N+1}$ of $Y\subset\p^N$ we necessarily have $\#(\Sing(X))<\infty$ since $X$ is smooth along the very ample divisor $Y=X\cap H$.
We also remark that in our definition $Y$ is a fixed hyperplane section. In the classical approach
usually it was required that $H$ was a general hyperplane section of $X$, see for example \cite{Scorza1}.
Under these more restrictive hypotheses one  can always suppose that a general point on $Y$ is also a  general point on $X$.
\medskip

\subsection{Extensions of $\mathcal L_{x,Y}\subset\p^{n-1}$ via $\mathcal L_{x,X}\subset\p^n$}

Let $y\in Y$ be a general point and let us consider an extension $X\subset\p^{N+1}$ of $Y$ and an irreducible component $\mathcal L_{y,Y}^j$ of $\mathcal L_{y,Y}\subset\p^{n-1}$,
which is  a smooth irreducible variety by Proposition \ref{Yx}. The results of \S 1 yield that this property is immediately translated in terms of Hilbert schemes of lines. Indeed we deduce the following result, where part (4) requires an ad hoc proof
since in our hypotheses the point $y\in Y$ is general on $Y$, but not necessarily on $X$, so that we cannot apply
Proposition \ref{Yx}.
\medskip

\begin{Proposition}\label{extlines} Let $X\subset\p^{N+1}$ be an irreducible projective variety which is an extension of
the non-degenerate manifold $Y\subset\p^N$. Let $n=\dim(Y)\geq 1$ and let $y\in Y$ be an arbitrary point such that
$\mathcal L_{y,Y}\neq\emptyset$. Then: 
\begin{enumerate}
\item $\mathcal L_{y,X}\cap \p((t_yY)^*)=\mathcal L_{y,Y}$ as schemes.

\item if $y\in Y$ is general, then $\dim_{[L]}(\mathcal L_{y,X})=\dim_{[L]}(\mathcal L_{y,Y})+1$ and $[L]$ is a smooth
point of $\mathcal L_{y,X}$ for every $[L]\in\mathcal L_{y,Y}.$

\item if $y\in Y$ is general and if $\mathcal L_{y,Y}^j$ is an irreducible component of positive dimension,
then there exists an irreducible component $\mathcal L_{y,X}^j$ such that $\mathcal L_{y,Y}^j=\mathcal L_{y,X}^j\cap\p((t_yY)^*)$
as schemes.

\item If $y\in Y$ is general, then $\Sing(\mathcal L_{y,X})\subseteq S_{y,X}$.
\end{enumerate}

\end{Proposition}
\begin{proof} Let  $Y=X\cap H$, with $H=\p^N\subset\p^{N+1}$ a hyperplane and let notation be as in subsection \ref{equations}.
The conclusion in (1) immediately follows from \eqref{eqLxE}.

Let us pass to (2) and consider an arbitrary  line $[L]\in \mathcal L_{y,Y}^j$, an irreducible component of the smooth not necessarily irreducible variety $\mathcal L_{y,Y}$.
 We have an exact sequence
of normal bundles
\begin{equation}\label{tangentspace}
0\to N_{L/Y}\to N_{L/X}\to N_{Y/X|L}\simeq\O_{\p^1}(1)\to 0.
\end{equation}

Since $y\in Y$ is general, $N_{L/Y}$ is generated by global sections, see the proof of Proposition \ref{Yx},
so that  \eqref{split} yields

\begin{equation}\label{normale}
N_{L/X}\simeq N_{L/Y}\oplus\O_{\p^1}(1)\simeq \O_{\p^1}(1)^{s(L,Y)+1}\oplus\O_{\p^1}^{n-s(L,Y)-1}.
\end{equation}

Thus also $N_{L/X}$ is generated by global sections,  $\mathcal L_{y,X}$ is smooth at $[L]$ and 
$\dim_{[L]}(\mathcal L_{y,X})=\dim_{[L]}(\mathcal L_{y,Y})+1$, proving (2). 

Therefore if $y\in Y$ is general,  
there exists a unique irreducible component of $\mathcal L_{y,X}\subset\p((t_yX)^*)$, let us say $\mathcal L_{y,X}^j$, containing $[L]$ and by the previous calculation
$\dim(\mathcal L^j_{y,X})=s(L,Y)+1=\dim(\mathcal L_{y,Y}^j)+1$. Recall that by part (1) we have $t_{[L]}\mathcal L_{y,Y}=t_{[L]}\mathcal L_{y,X}\cap \p((t_yY)^*)$ so that
\begin{equation}\label{extYx}
\mathcal L_{y,Y}^j\subseteq \mathcal L_{y,X}^j\cap \p((t_yY)^*)\subseteq \mathcal L_{y,Y}\subset \p^{n-1}=\p((t_yY)^*),
\end{equation}
yielding that $\mathcal L_{y,Y}^j$ is an irreducible component of $\mathcal L_{y,X}^j\cap \p((t_yY)^*)$ as well as an irreducible component of the smooth variety $\mathcal L_{y,Y}$.
Hence,  if $\dim(\mathcal L_{y,Y}^j)\geq 1$, we have the equality $\mathcal L^j_{y,Y}= \mathcal L_{y,X}^j\cap \p((t_yY)^*)$ as schemes, 
i.e. under this hypothesis $\mathcal L_{y,X}^j\subset\p((t_yX)^*)$ (or better $(\mathcal L_{y,X}^j)_{\red}$) is   a projective  extension of the smooth positive dimensional irreducible variety $\mathcal L_{y,Y}^j\subset\p((t_xY)^*)$.
Indeed, $\dim(\mathcal L^j_{y,Y})\geq 1$ forces $\dim(\mathcal L_{y,X}^j)\geq 2$ so that it is sufficient to recall that $\mathcal L_{y,X}$ is smooth along $\mathcal L_{y,Y}$ by the previous discussion and also that an
arbitrary hyperplane section of the irreducible variety $(\mathcal L_{y,X}^j)_{\red}$ is connected by the Fulton-Hansen Theorem, \cite{FH}.  More precisely, if $\dim(\mathcal L_{y,Y}^j)\geq 1$, then equality as schemes holds in \eqref{extYx}, proving part (3).

By  \cite[Proposition 4.9]{Debarre} there exists a non-empty  open subset $U\subseteq X$ such that
$N_{\widetilde L/X}$ is generated by global sections for every  line $\widetilde L\subset X_{\reg}$ 
intersecting $U$. If $U\cap Y\neq \emptyset$, then (4) clearly holds. Suppose $Y\cap U=\emptyset$.
Let  $[\widetilde L]\in \mathcal L_{y,X}\setminus S_{y,X}$. If $\widetilde L\cap U\neq \emptyset$, then $[\widetilde L]$ is
a smooth point of $\mathcal L_{y,X}$ by the previous analysis. If $\widetilde L\cap U=\emptyset$, then $\widetilde L\subset Y$
by the generality of $y\in Y$ and $N_{\widetilde L/X}$ is generated by global sections by \eqref{normale},
concluding the proof of (4).
\end{proof}
\medskip

Now we are in position to prove  the main result of this section and  to deduce some  applications.
\medskip

\begin{Theorem}\label{criterion}
Let notation be as above and let $y\in Y$ be a general point. Then:
\begin{enumerate}

\item Suppose there exist two   distinct irreducible components $\mathcal L_{y,X}^1$ and $\mathcal L_{y,X}^2$ of $\mathcal L_{y,X}\subset\p((t_yX)^*)$, extending two irreducible components $\mathcal L_{y,Y}^1$, respectively $\mathcal L_{y,Y}^2$, of $\mathcal L_{y,Y}$ in the sense specified above. If $\mathcal L_{y,X}^1\cap \mathcal L_{y,X}^2\neq\emptyset$, 
then $X\subset\p^{N+1}$ is a cone over $Y\subset\p^N$ of vertex a point $p\in \p^{N+1}\setminus \p^N$.

\item If $\mathcal L_{y,Y}\subset\p((t_yY)^*)$ is a manifold whose extensions are singular, then
every extension of $Y\subset\p^N$ is trivial.
\end{enumerate}
\end{Theorem}
\begin{proof} By the above discussion, we get that in both cases, for $y\in Y$ general,  the variety $S_{y,X}\subseteq \mathcal L_{y,X}$ is not empty so that for $y\in Y$ general  there exists a line $L_y\subseteq X$ passing through $y$ and through a singular point  $p_y\in L_y\cap \Sing(X)$. Since $Y$ is irreducible and since $\Sing(X)$ consists of a finite 
number of points, there exists $p\in \Sing(X)$ such that $p\in L_y$ for $y\in Y$ general. This implies
that $X=S(p,Y)$ is a cone over $Y$ with vertex $p$. 
\end{proof}
\medskip

The first easy consequence is a result due to Scorza (see \cite{Scorza2} and also \cite{Zakdual}, \cite{Badescu}), proved by him under the stronger assumption that $Y=X\cap H$ is a general hyperplane section of $X$.
Under these more restrictive hypotheses, the analysis before the proof of Theorem \ref{criterion} could  be simplified via Proposition \ref{Yx}, since we may  assume that the general point
$y\in Y$ is also general on $X$.
\medskip

\begin{Corollary}\label{Segre} Let $1\leq a\leq b$ be integers,  let $n=a+b\geq 3$ and let
$Y\subset\p^{ab+a+b}$ be a smooth irreducible variety projectively equivalent to the
Segre embedding $\p^a\times\p^b\subset\p^{ab+a+b}$. Then every extension 
of $Y$ in $\p^{ab+a+b+1}$ is trivial.
\end{Corollary}
\begin{proof} For $y\in Y$ general, it is well known that $\mathcal L_{y,Y}=\mathcal L_{y,Y}^1\amalg \mathcal L_{y,Y}^2\subset\p^{a+b-1}=\p^{n-1}$ with $\mathcal L_{y,Y}^1=\p^{a-1}$ and $\mathcal L_{y,Y}^2=\p^{b-1}$, both linearly embedded. Observe that
$b-1\geq 1$.  By \eqref{extYx} and the discussion following it,
there exist two irreducible components  $\mathcal L_{y,X}^j$, $j=1,2$,  of $\mathcal L_{y,X}\subset\p^{n}=\p^{a+b}$ with $\dim(\mathcal L_{y,X}^1)=a$ and $\dim(\mathcal L_{y,X}^2)=b$. If $a\neq b$ then clearly $\mathcal L_{y,X}^1\neq \mathcal L_{y,X}^2$. If $a=b\geq 2$, then $\mathcal L_{y,X}^1\neq \mathcal L_{y,X}^2$ because an arbitrary hyperplane section of a variety of dimension at least
2 is connected, see \cite{FH}.
Since $a+b=n$, $\mathcal L_{y,X}^1\cap \mathcal L_{y,X}^2\neq\emptyset$ and the conclusion follows from the first part of Theorem \ref{criterion}.
\end{proof}
\medskip

The previous result
has some interesting consequences via iterated
applications of the second part of Theorem \ref{criterion}. Indeed, let us consider the following homogeneous varieties 
(also known as 
irreducible hermitian symmetric spaces), in their homogeneous embedding, and the description of the Hilbert scheme
of lines passing through a general point, see \cite[\S 1.4.5]{Hwang} and also \cite{Strick}. 
\medskip

\begin{equation}\label{hermitian}
\begin{tabular}{|c|c|c|c|}
\hline  &$Y$   &  $\mathcal L_{y,Y}$   &   $\tau_y:\mathcal L_{y,Y}\to \p((t_yY)^*)$\\
\hline  1 &$\mathbb G(r,m)$   & $\p^r\times\p^{m-r-1}$  & \text{Segre embedding}\\
\hline  2 &$SO(2r)/U(r)$& $\mathbb G(1,r-1)$ & \text{Pl\" ucker embedding}\\
\hline  3 &$E_6$ & $SO(10)/U(5)$& \text{miminal embedding}\\
\hline  4 &$E_7/E_6\times U(1)$   & $E_6$  & \text{Severi embedding}\\
\hline  5 &$Sp(r)/U(r)$ & $\p^{r-1}$ & \text{quadratic Veronese embedding}\\
\hline
\end{tabular}
\end{equation}
\medskip

There are also the following homogeneous contact manifolds with Picard number one associated
to a complex simple Lie algebra $\mathbf g$, whose Hilbert scheme of lines passing through
a general point is known. Let us observe that in these examples the variety 
$\mathcal L_{y,Y}\subset\p^{n-1}=\p((t_yY)^*)$ is degenerate and its linear span is exactly $\p((D_y)^*)=\p^{n-2}$, there $D_y$ is the tangent
space at $y$ of the distribution associated to the contact structure on $Y$, i.e. there is the following factorization $\tau_y:\mathcal L_{y,Y}\to \p((D_y)^*)\subset\p((t_yY)^*)$. For more details one can consult \cite[\S 1.4.6]{Hwang}.\medskip

\begin{equation}\label{contact}
\begin{tabular}{|c|c|c|c|}
\hline  &$\mathbf g$   &  $ \mathcal L_{y,Y}$   &   $\tau_x:\mathcal L_{x,Y}\to \p((D_y)^*$\\
\hline   6 &$F_4$   & $Sp(3)/U(3)$  & \text{Segre embedding}\\
\hline  7 &$E_6$ & $\mathbb G(2,5)$& \text{Pl\" ucker  embedding}\\
\hline  8 &  $E_7$ & $SO(12)/U(6)$  & \text{minimal  embedding}\\
\hline  9 &$E_8$ & $E_7/E_6\times U(1)$& \text{minimal  embedding}\\
\hline  10&${\mathbf so}_{m+4}$& $\p^1\times Q^{m-2}$ & \text{Segre embedding}\\
\hline
\end{tabular}
\end{equation}

\medskip

 By case 1')
we shall denote a variety as in 1) of \eqref{hermitian} satisfying the following numerical conditions:
$r<m-1$; if  $r=1$, then $m\geq 4$. By 2') we shall denote a variety as in 2) with $r\geq 5$. 
\medskip

\begin{Corollary}\label{exthermitian} Let $Y\subset\p^N$ be a manifold as in  Examples 1'), 2'), 3), 4), 7),
8), 9) above. Then every extension of $Y$ is trivial.
\end{Corollary}
\begin{proof} In cases 2'), 3), 4) and 9) in the statement the variety $\mathcal L_{y,Y}\subset\p^{n-1}$ of one example is the variety $Y\subset\p^N$ occurring in the next one. Thus for these cases, by the second part of Theorem \ref{criterion},  it is sufficient to prove the result for case 1'). For this variety the conclusion follows from Corollary \ref{Segre}. For the remaining cases, the variety $\mathcal L_{y,Y}\subset\p^{n-1}$ is either
as in case 1') with $(r,m)=(2,5)$ or as in case 2) with $r=6$ and the conclusion follows once
again by the second part of Theorem \ref{criterion}.
\end{proof} 
\medskip

The next result is also classical and well-known but we provide a direct geometric proof. Under the assumption that the hyperplane section $H\cap X=Y$ is general, it was proved  by C. Segre for $n=2$ in \cite{CSegre} and by Scorza in \cite{Scorza1}, see also \cite{Terracini},
for arbitrary $n\geq 2$ (and also for arbitrary Veronese embeddings $\nu_d(\p^n)\subset\p^{N(d)}$, with $n\geq 2$ and $d\geq 2$; modern proofs of this general case are contained in \cite{Badescu} and in \cite{Zakdual}).
\medskip

\begin{Proposition}\label{Veronese} Let $n\geq 2$ and let $Y\subset\p^{\frac{n(n+3)}{2}}$ be
a manifold projectively equivalent to the quadratic Veronese embedding 
$\nu_2(\p^n)\subset\p^{\frac{n(n+3)}{2}}$. Then every extension of $Y$ is trivial.
\end{Proposition}
\begin{proof} Let $y\in Y$ be a general point and let $N=\frac{n(n+3)}{2}$. Since $\mathcal L_{y,Y}=\emptyset$, then $\mathcal L_{y,X}\subset\p^n$, if
not empty, consists of at most a finite number of points and through $y\in X$ there passes at most a finite number of
lines contained in $X$. Consider a conic $C\subset Y$ passing through $y$. Then $N_{C/Y}\simeq\O_{\p^1}(1)^{n-1}$. 
The exact sequence
of normal bundles
$$0\to N_{C/Y}\to N_{C/X}\to N_{Y/X|C}\simeq\O_{\p^1}(2)\to 0,$$
yields
$$N_{C/X}\simeq N_{C/Y}\oplus\O_{\p^1}(2)\simeq \O_{\p^1}(1)^{n-1}\oplus\O_{\p^1}(2).$$

Thus there exists a unique irreducible component $\mathcal C_{y,X}$ of the Hilbert scheme of conics contained in $X\subset\p^{N+1}$ passing through $y\in X$ to which  $[C]$ belongs. Moreover $\dim(\mathcal C_{y,X})=n+1$ and the conics
parametrized by $\mathcal C_{y,X}$ cover $X$. Hence there exists a one dimensional  family of conics
through $y$  and  a general point $x\in X$. By Bend and Break, see for example \cite[Proposition 3.2]{Debarre}, there is at least a singular conic through $y$ and $x$. Since $X\subset\p^{N+1}$ is not a linear space, there exists no line joining $y$ and a general $x$, i. e. the  singular  conics through $x$ and $y$ are  reduced. Thus given a general point $x$ in $X$, there
 exists a line $L_x\subset X$ through $x$, not passing through $y$,  and a line $L_y\subset X$ through $y$ such that
 $L_y\cap L_x\neq \emptyset$. Since there are a finite number of lines contained in $X$ and passing through $y$, we can conclude that given a general point $x\in X$, there exists a fixed line passing through $y$, $\widetilde{L}_y$,
 and a line $L_x$ through $x$  such that $L_x\cap  \widetilde{L}_y\neq \emptyset.$
 
Moreover, a general conic  $[C_{x,y}]\in\mathcal C_{y,X}$ and passing through a general point $x$ is irreducible, does not pass through the finite set $\Sing(X)$ and has ample normal bundle 
verifying  $h^0(N_{C_{x,y}/X}(-1))=h^0(N_{C/X}(-1))=n+1$. This means that the deformations of $C_{x,y}$ keeping $x$ fixed cover an open subset of $X$ and also that through general
points $x_1,x_2\in X$ there passes a one dimensional family of irreducible conics. The plane spanned by one of these conics contains $x_1$ and $x_2$ so that it has to vary
with the conic. Otherwise the fixed plane would be contained in $X$ and $X\subset\p^{N+1}$ would be a linearly embedded $\p^{N+1}$, which is contrary to our assumptions. In conclusion
through a general point $z\in <x_1,x_2>$ there passes at least a one dimensional family of secant lines to $X$ so that 
\begin{equation}\label{dimSX}
\dim(SX)\leq 2(n+1)-1=2n+1<N+1=\frac{n(n+3)}{2}+1,
\end{equation}
yielding $SX\subsetneq\p^{N+1}$.
 
Suppose the  point $p_x=\widetilde{L}_y\cap L_x$, for $y\in Y$ general, varies on  $\widetilde{L}_y$. Then the linear span of two general tangent spaces $T_{x_1}X$ and $T_{x_2}X$ would
contain the line $\widetilde{L}_y$. Since $T_zSX=<T_{x_1}X,T_{x_2}X>$ by the Terracini Lemma, we deduce that a general tangent space to $SX$ contains
$\widetilde{L}_y$ and a fortiori $y$. Since $SX\subsetneq\p^{N+1}$, the variety $SX\subset\p^{N+1}$ would be  a cone whose vertex, which is a linear space,
 contains $\widetilde{L}_y$ and a fortiori $y\in Y$. By the generality of $y\in Y$ we would deduce that $Y\subset\p^N$ is degenerate.

Thus $p_x=\widetilde{L}_y\cap L_x$ does not vary with $x\in X$ general. Let us denote this point by
$p$. Then clearly $X\subset\p^{N+1}$ is a cone with vertex $p$ over $Y$.
\end{proof}
\medskip

\begin{Corollary}\label{extcontact} Let $Y\subset\p^N$ be a manifold  either as in 5) above with $r\geq 3$ or as in 6)
above. Then every extension of $Y$ is trivial.
\end{Corollary}
\begin{proof} By \eqref{hermitian} we know that in case 5) with $r\geq 3$ we have $n-1=\frac{(r-1)(r+2)}{2}$ and the variety $\mathcal L_{y,Y}\subset\p^{n-1}$
is projectively equivalent to $\nu_2(\p^{r-1})\subset\p^{\frac{(r-1)(r+2)}{2}}$. To conclude we apply Proposition \ref{Veronese} and the second part of Theorem \ref{criterion}. Case 6) follows
from case 5) with $r=3$ by the second part of Theorem \ref{criterion}.
\end{proof}
\medskip

\begin{Remark}\label{dualhom} There is a different and interesting approach to Corollary \ref{exthermitian} and Corollary \ref{extcontact}
based on the theory of dual varieties and proposed by Zak in \cite{Zakdual}, which also avoids direct computations
of vanishing of cohomology groups in each case. This approach is less direct and less elementary
than ours and it is based on the following facts. By a result of Kempf the dual variety of any homogeneous variety is normal, see e.g. \cite[Theorem III.1.2]{Zak}. Then in \cite[Corollary 1]{Zakdual} it is stated that a smooth variety $X\subset\p^N$ whose dual variety is normal admits only trivial
extensions. As far as we know, to establish this result one first shows that the  normality of $X^*$ implies its linear normality, which seems to follow from  some well known but not trivial results. Finally  one applies  \cite[Theorem ]{Zakdual}, which is a general criterion for admitting only trivial extension. For us Theorem \ref{criterion} is  simply another incarnation of the Principle described in the Introduction while  Corollary \ref{exthermitian} and Corollary \ref{extcontact}, surely well known to everybody, were included only to show that they are an immediate consequences of Scorza's result  in \cite{Scorza2}, a fact which seems to have been overlooked till now.

\end{Remark}

\def\bibaut#1{{\sc #1}}

\end{document}